\documentclass[12pt]{article} 
\usepackage{cmap}
\usepackage{ucs}
\usepackage[utf8]{inputenc}
\usepackage[top=2.5cm,bottom=2.5cm,left=2.5cm,right=2.5cm]{geometry}
\usepackage{amsfonts,amsmath,amsthm,amscd,amssymb,latexsym}
\usepackage{mathtools}
\usepackage[english]{babel}
\usepackage[linkcolor=blue,colorlinks=true]{hyperref}

\usepackage{tikz}
\usetikzlibrary{arrows.meta,positioning,calc,decorations.pathreplacing}

\newtheorem{theorem}{Theorem}[section]

\newtheorem{proposition}[theorem]{Proposition}
\newtheorem{corollary}[theorem]{Corollary}

\theoremstyle{definition}

\mathtoolsset{showonlyrefs=true}

\title{\bf A non-existence result for vertex-girth-regular graphs}

\author{
Jorik Jooken \thanks{Department of Computer Science, KU Leuven Kulak, 8500 Kortrijk, Belgium.
}
\and Denys Lohvynov
 \thanks{National Technical University of Ukraine “Igor Sikorsky Kyiv Polytechnic Institute”, 37 Beresteiskyi Ave.,
03056 Kyiv, Ukraine.\\ Email addresses:
 \protect\href{mailto:jorik.jooken@kuleuven.be}{\protect\nolinkurl{jorik.jooken@kuleuven.be}} and
\protect\href{mailto:denislogvinov410@gmail.com}{\protect\nolinkurl{denislogvinov410@gmail.com}}}
}

\begin{document}

\maketitle

\begin{abstract}
A $k$-regular graph of girth $g$ is called \textit{vertex-girth-regular} if every vertex is contained in the same number of cycles of length $g$. For integers $n, k, g$ and $\lambda$, we denote such a graph on $n$ vertices in which every vertex lies on exactly $\lambda$ cycles of length $g$ by a $\text{vgr}(n,k,g,\lambda)$-graph. It is well-known that any vertex-girth-regular graph satisfies $\lambda \le \frac{k(k-1)^{\left\lfloor \frac{g}{2} \right\rfloor}}{2}$. Graphs for which $\lambda$ is close to this bound are of particular interest in connection with the cage problem, since requiring many girth cycles through every vertex is a natural way to isolate highly structured candidates for small regular graphs of prescribed girth. In this paper, we prove that for every $k\ge 3$ and every integer $0< \varepsilon \leq \frac{k-1}{2}$, there does not exist a $\text{vgr}(n,k,5,\frac{k(k-1)^2}{2}-\varepsilon)$-graph. Previous non-existence results had already settled all odd girths at least $7$ and very recently also girth $3$, leaving girth $5$ as the only girth for which no non-trivial non-existence result was known. Thus, our result resolves the final remaining case and completes the picture for odd girths.

			\vskip 3mm
		
			\end{abstract}
            
\section{Introduction}

One of the central problems in extremal graph theory is the well-known \textit{cage problem}, which asks for given integers $k, g \geq 3$ for the order of the smallest $k$-regular graph of girth $g$. This smallest order is denoted by $n(k,g)$ and the corresponding extremal graphs are called \textit{$(k,g)$-cages}. This problem is notoriously difficult in general as the precise value of $n(k,g)$ is only known for a few small integers $k$ and $g$ as well as for infinitely many integers $k$ if $g \in \{3,4,6,8,12\}$~\cite{EJ12}. As a result, much of the attention in the literature has shifted towards establishing strong upper bounds for $n(k,g)$ by finding small $k$-regular graphs of girth $g$ (see e.g.~\cite{AABB17,AABL12,EGJSTV25,EJ12,LUW95}). For many small values of $k$ and $g$, the $(k,g)$-cages exhibit a lot of structure. Therefore, a useful heuristic to find small $k$-regular graphs of girth $g$ is to limit the search space by only considering graphs that exhibit such a structure. For example, edge-transitive and vertex-transitive graphs received a lot of attention both from a theoretical as well as a computational point of view, and this resulted in complete censuses of such graphs~\cite{HR20,J25,EncyclopediaOfGraphs,PSV13,PSV15}. While several of the smallest known $k$-regular graphs of girth $g$ come from such censuses, the heuristic of only considering edge-transitive or vertex-transitive graphs is sometimes too restrictive.

This observation was one of the motivations for Jajcay, Kiss and Miklavi{\v{c}}~\cite{JKM18} to introduce a new class of graphs called \textit{edge-girth-regular} graphs in 2018. For integers $n, k, g$ and $\lambda$, an $\text{egr}(n,k,g,\lambda)$-graph ($\text{egr}$ stands for edge-girth-regular) is a $k$-regular graph of girth $g$ on $n$ vertices such that each edge is contained in $\lambda$ cycles of length $g$. It is worth noting that every edge in an edge-transitive graph is contained in the same number of cycles of any length, while an edge-girth-regular graph is not necessarily edge-transitive. It is then natural to consider a new extremal question: For given integers $k, g$ and $\lambda$, what is the order $n$ of a smallest $\text{egr}(n,k,g,\lambda)$-graph? However, before posing this question, it is first necessary to ask whether such a graph exists at all. While $k$-regular graphs of girth $g$ exist for all integers $k, g \geq 3$~\cite{S63}, the same conclusion does not hold for edge-girth-regular graphs for all parameters $\lambda$. Indeed, Jajcay, Kiss and Miklavi{\v{c}}~\cite{JKM18} showed the following.

\begin{proposition}[Jajcay, Kiss and Miklavi{\v{c}}~\cite{JKM18}]
    Let $G$ be an $\text{egr}(n,k,g,\lambda)$-graph. Then $\lambda \leq (k-1)^{\lfloor \frac{g}{2}\rfloor}$.
\end{proposition}

For integers $k, g$ and $\lambda$, let $n(k,g,\lambda)$ be the smallest order $n$ for which an $\text{egr}(n,k,g,\lambda)$ graph exists (or $\infty$ if no such graph exists). It is worth noting that for many integers $k, g$ and $\lambda$, we have $n(k,g,\lambda) \geq n(k,g,\lambda+1)$~\cite{GJ25}. Therefore, with respect to the cage problem, it is interesting to understand $n(k,g,\lambda)$ when $\lambda$ is close to the upper bound $(k-1)^{\lfloor \frac{g}{2}\rfloor}$. Jajcay, Kiss and Miklavi{\v{c}}~\cite{JKM18} established that an $\text{egr}(n,k,g,(k-1)^{\lfloor \frac{g}{2}\rfloor})$-graph can only exist in a very special case, namely if it is a Moore graph\footnote{The Moore bound $M(k,g)$ provides a lower bound for $n(k,g)$ and graphs achieving this bound are called Moore graphs. It is known that if $G$ is a $k$-regular Moore graph of girth $g$, then $k=2$ or $g \in \{3,4,6,8,12\}$ or $(k,g) \in \{(3,5),(7,5),(57,5)\}$.}.

The next natural thing is to investigate the situation where $\lambda$ is close to $(k-1)^{\lfloor \frac{g}{2}\rfloor}$, but strictly smaller. This situation is partially understood for two classes of graphs that are more general than edge-girth-regular graphs (and we complete the picture in the current paper). We will now introduce these two graph classes. Let $G$ be a $k$-regular graph of girth $g$ and for an edge $e \in E(G)$, let $n(e)$ denote the number of cycles of length $g$ that contain $e$. For a vertex $v \in V(G)$, let $e_1, e_2, \ldots, e_k$ be the $k$ edges incident with $v$ ordered such that $n(e_1) \leq n(e_2) \leq \ldots \leq n(e_k)$. Here, $(n(e_1), n(e_2), \ldots, n(e_k))$ is called the \textit{signature} of $v$. Using this notation, $G$ is called edge-girth-regular if $n(e)=n(e')=\lambda$ for all edges $e, e' \in E(G)$. The notion of edge-girth-regularity was generalized by Poto{\v{c}}nik and Vidali~\cite{PV19} as follows: we say that $G$ is \textit{girth-regular} if each pair of vertices in $G$ has the same signature. In turn, the notion of girth-regularity was further generalized by Jajcay, Jooken and Porups{\'a}nszki~\cite{JJP24} as follows: we call $G$ \textit{vertex-girth-regular} if every pair of vertices is contained in the same number of cycles of length $g$, i.e., the sums of the entries in the signature are equal. Similarly to edge-girth-regular graphs, we say that a $\text{vgr}(n,k,g,\lambda)$ graph is a $k$-regular graph of girth $g$ on $n$ vertices such that each vertex is contained in $\lambda$ cycles of length $g$. It is known for a girth-regular graph that again $n(e_k) \leq (k-1)^{\lfloor \frac{g}{2}\rfloor}$~\cite{PV19}, whereas for a vertex-girth-regular graph we have $\sum_{i=1}^{k}n(e_i) \leq k(k-1)^{\lfloor \frac{g}{2}\rfloor}$~\cite{JJP24}, i.e., every vertex is contained in at most $\frac{k(k-1)^{\lfloor \frac{g}{2}\rfloor}}{2}$ cycles of length $g$.

Kiss, Miklavi{\v{c}} and Sz{\H{o}}nyi~\cite{KMS22} established a non-existence result for girth-regular graphs of even girth.

\begin{theorem}[Kiss, Miklavi{\v{c}} and Sz{\H{o}}nyi~\cite{KMS22}]

Let $k, g \geq 3$ and $0 < \varepsilon < k-1$ be integers. There does not exist a $k$-regular girth-regular graph of even girth $g$ with signature $(n(e_1),n(e_2),\ldots,n(e_k))$ such that $n(e_k)=(k-1)^{\frac{g}{2}}-\varepsilon$. 
\end{theorem}

In a similar vein, Jajcay, Jooken and Porups{\'a}nszki~\cite{JJP24} established such a non-existence result for vertex-girth-regular graphs, but this time only for odd girths that are at least $7$.

\begin{theorem}[Jajcay, Jooken and Porups{\'a}nszki~\cite{JJP24}]
Let $k\geq 3$, $g \geq 7$ and $0 < \varepsilon \leq \frac{k-1}{2}$ be integers. There does not exist a $\text{vgr}(n,k,g,\frac{k(k-1)^{\lfloor \frac{g}{2}\rfloor}}{2}-\varepsilon)$-graph where $g$ is odd.
\end{theorem}

This leaves the cases $g=3$ and $g=5$ open. Recently, the case $g=3$ was solved by Hak, Kozerenko, Lohvynov and Yarosh~\cite{HKLY26}. In the current paper, we solve the final remaining case $g=5$ by a careful structural analysis of the number of cycles of length $5$ containing appropriately chosen vertices in a $\text{vgr}(n,k,5,\lambda)$-graph. Hence, this leads to the following complete picture for odd girths.

\begin{corollary}
Let $k\geq 3$, $g \geq 3$ and $0 < \varepsilon \leq \frac{k-1}{2}$ be integers. There does not exist a $\text{vgr}(n,k,g,\frac{k(k-1)^{\lfloor \frac{g}{2}\rfloor}}{2}-\varepsilon)$-graph where $g$ is odd.
\end{corollary}

\section{The result}

\begin{theorem}
Let $k\geq 3$ and $0 < \varepsilon \leq \frac{k-1}{2}$ be integers. There does not exist a $\text{vgr}(n,k,5,\frac{k(k-1)^{2}}{2}-\varepsilon)$-graph.
\end{theorem}

\begin{proof}
    Let us assume that there is such a graph $G$ for the sake of obtaining a contradiction. 
    
    We can readily note that $\forall v \in V(G): |N_2 (v)| = k(k-1)$ and the number of $5$-cycles containing a vertex $v$ is exactly the number of edges within $N_2(v)$, i.e. $|E(N_2(v), N_2(v))|$, since $G$ has girth $5$. We can also note that $\forall u \in V(G):$
    \begin{equation}\label{eq:outerEdges}
     |E(N_2(u), V\setminus(N(u)\cup N_2(u)))| = 2\varepsilon,   
    \end{equation}
     since
    \begin{multline}
      (k-1)\underbrace{|N_2(u)|}_{k(k-1)} = \sum_{w \in N_2(u)} |E(\{w\}, V(G)\setminus N(u))| \\ = 2\underbrace{|E(N_2(u), N_2(u))|}_{\frac{k(k-1)^2}{2}-\varepsilon} + |E(N_2(u), V\setminus(N(u)\cup N_2(u)))|.  
    \end{multline}

        Let us start by proving that 
        \begin{equation}\label{eq:main_property}
    \forall u \in V(G) \forall v', v'' \in N_2(u) (v'\neq v''): N(v')\cap N(v'') \subset N(u) \cup N_2(u).        
        \end{equation}

    Let us assume that it is not the case; then there must be a vertex $u \in V(G)$ as well as $v', v'' \in N_2(u)$ and $v \in V(G)\setminus (N(u)\cup N_2(u))$ such that $v'v, v''v \in E(G)$.

    What we want to do now is to bound from above the number of $5$-cycles containing $v$, denoting this number by $Y$. To that end, we shall introduce several sets, namely
    \begin{align}
     V_A &= N_2(v)\cap N(u),\\
     V_B &= N_2(v)\cap N_2(u),\\
     V_C &= N_2(v)\setminus(N(u) \cup N_2(u)).
    \end{align}
    Note that $|V_A|\geq 2$. Thus, we can see that
    \begin{multline}
      Y = |E(N_2(v), N_2(v))|  = \underbrace{|E(V_A, V_A)|}_{0} + |E(V_A, V_B)| \\ + \underbrace{|E(V_A, V_C)|}_{0} + |E(V_B, V_B)| + |E(V_B, V_C)| + |E(V_C, V_C)| \\  =  |E(V_A, V_B)|  + |E(V_B, V_B)| + |E(V_B, V_C)| + |E(V_C, V_C)|.   
    \end{multline}

    Also, note that 
    \begin{multline}\label{eq:V_C_induced}
     (k-1)|V_C| = \sum_{w \in V_C} |E(\{w\}, V(G)\setminus N(v))| \geq \sum_{w \in V_C} |E(\{w\}, V_C  \cup V_B)| \\ = 2 |E(V_C, V_C)| + |E(V_B, V_C)|   
    \end{multline}
    and that 
    \begin{multline}\label{eq:V_B_induced}
     (k-1)|V_B| = \sum_{w \in V_B} |E(\{w\}, V(G)\setminus N(v))| \\ \geq \sum_{w \in V_B} |E(\{w\}, V_C  \cup V_B \cup V_A)| \\ 
     = |E(V_A, V_B)| + 2 |E(V_B, V_B)| + |E(V_B, V_C)|.   
    \end{multline}

Now, let us focus on $|E(V_A, V_B)|$. To that end we shall introduce several new notions:
\begin{itemize}
    \item $N(u) = \{u_1, ..., u_k\}$, and for each $i \in \{1,2,\ldots,k\}$, we define $V_2(i) = N(u_i) \setminus \{u\}$; 
    \item $ d_i = |\{e = x y \in E: x \in V_2(i), y \in N(v) \setminus N_2(u) \}|$, namely $d_i$ is the number of edges connecting $V_2(i)$ and a vertex connected to $v$ but not in $N_2(u)$;
    \item  $a_i = 1_{N(v) \cap V_2 (i) \neq \emptyset}$, namely $a_i$ indicates whether there is an edge from $v$ to $V_2(i)$. Note that if there is such an edge, it is unique because $G$ has girth $5$.
\end{itemize}

\begin{figure}[ht]
\centering
\begin{tikzpicture}[
    x=1.3cm,y=1cm,
    block/.style={draw,rounded corners,minimum width=1.0cm,minimum height=.55cm},
    dot/.style={circle,fill,inner sep=1.4pt},
    >=Latex,scale=0.72
]
    \node[dot,label=above:{\large $u$}] (u) at (0,0) {};

    \node[dot,label=above left:{\large $u_1$}] (uPrime) at (-3,-2) {};
    \node[dot,label=above left:{\large $u_2$}] (u1) at (0,-2) {};
    \node[dot,label=above right:{\large $u_3$}] (u2) at (3,-2) {};

    \node[dot,label=above left:{\large $v'$}] (vPrime) at (-5,-4) {};
    \node[dot] (v12) at (-3,-4) {};

    \node[dot,label=above left:{\large $v''$}] (v21) at (-1,-4) {};
    \node[dot] (v22) at (1,-4) {};

    \node[dot] (v31) at (3,-4) {};
    \node[dot] (v32) at (5,-4) {};
    
    \draw (u) -- (uPrime);
    \draw (u) -- (u1);
    \draw (u) -- (u2);

    \draw (uPrime) -- (vPrime);
    \draw (uPrime) -- (v12);

    \draw (u1) -- (v21);
    \draw (u1) -- (v22);

    \draw (u2) -- (v31);
    \draw (u2) -- (v32);

    \node[dot,label=above:{\large $v$}] (v) at (-4,-6) {};
    \draw (v) -- (vPrime);

    \node[dot] (v2) at (-2,-8) {};
    \draw (v) -- (v2);


    
    \node[dot] (v2Son1) at (-2,-10) {};
    \draw (v2) -- (v2Son1);


    \draw (v2) -- (v31);
    \draw (vPrime) to[bend right=22] (v22);

    \draw (v) -- (v21);

    \node[block, minimum width=2.7cm, minimum height=0.8cm, label=above:{\large $V_{B}$}] (D)  at (2,-4) {};


    \node[block, minimum width=1.0cm, minimum height=1.0cm, label=below:{\large $V_{C}$}] at (v2Son1) {};

    \node[block, minimum width=4.0cm, minimum height=1.3cm, label=above left:{\large $V_{A}$}] at (-1.7,-2) {};

    \node[block, minimum width=3.0cm, minimum height=1.2cm, label=above left:{\large $V_2(1)$}] at (-4.1,-3.9) {};
    

\end{tikzpicture}
\caption{An illustration of the notation used for the case where at least two vertices of $N_2(u)$ are adjacent to $v$. The displayed graph is a subgraph of $G$, so not all vertices and edges are visualized. If we assume that $k=3$, we can derive for this example that $d_1=0$, $d_2=0$, $d_3=1$, $a_1=1$, $a_2=1$ and $a_3=0$.}
\label{fig:illustrateNotation1}
\end{figure}

An example that illustrates these variables is given in Fig.~\ref{fig:illustrateNotation1}. Having introduced these variables, we want to note several identities.  

The first is $$|V_A| = \sum_i a_i,$$ which holds since $a_i = 1 $ if and only if $u_i \in N_2(v),$ for $u_i \in N(u)$. 

The second is 
\begin{equation}\label{eq:V_Ag}
  \sum_{i=1}^k d_i + |V_A| = \sum_{i=1}^k (d_i + a_i) \leq 2\varepsilon,  
\end{equation}
which holds since 
\begin{multline}
 2\varepsilon = |E(N_2(u), V\setminus(N(u)\cup N_2(u)))| = \sum_{i=1}^k |E(V_2(i), V\setminus(N(u)\cup N_2(u)))| \\ \geq \sum_{i=1}^k \underbrace{|E(V_2(i), N(v)\setminus N_2(u))|}_{d_i} + \underbrace{|E(V_2(i), \{v\})|}_{a_i}.    
\end{multline}
Since $a_i \in \{0,1\}$, we also have 
\begin{equation}\label{eq:2epsV_A}
    \sum_{i=1}^k d_i  a_i \leq 2\varepsilon - |V_A|.
\end{equation}

And finally, we have that
\begin{equation}\label{eq:E_AB}
|E(V_A, V_B)| \leq  \sum_{i=1}^k d_i  a_i + |V_A|(|V_A| - 1).    
\end{equation}
%
This inequality holds because any edge between $V_A$ and $V_B$ connects some $u_i \in V_A$ (where $a_i = 1$) to a vertex in $V_B \cap V_2(i)$. We can partition these endpoints based on the location of their unique neighbor in $N(v)$:
\begin{multline}
|E(V_A, V_B)| = \sum_{i=1}^k a_i |E(\{u_i\}, V_B\cap V_2(i))| = \sum_{i=1}^k a_i |V_B \cap V_2(i)| \\ 
= \sum_{i=1}^k a_i \Big( \underbrace{|\{w \in V_B \cap V_2(i) : N(w) \cap N(v) \not\subset N_2(u)\}|}_{\leq d_i} \\ 
+ \underbrace{|\{w \in V_B \cap V_2(i) : N(w) \cap N(v) \subset N_2(u)\}|}_{\leq |V_A| - 1} \Big).
\end{multline}

Returning to $Y$, by virtue of \eqref{eq:V_C_induced},~\eqref{eq:V_B_induced},~\eqref{eq:2epsV_A} and \eqref{eq:E_AB} we see that 
\begin{multline}
    2Y \leq (k-1)(\underbrace{|V_C| + |V_B|}_{k(k-1)-|V_A|}) + \sum_{i=1}^k d_i a_i + |V_A|(|V_A|-1) \\ \leq k(k-1)^2 -(k-1)|V_A| +2\varepsilon - |V_A| + |V_A|(|V_A| - 1).
\end{multline}
Note that from \eqref{eq:V_Ag} we can also see that $2 \leq |V_A| \leq 2\varepsilon$. Thus 
\begin{equation}
    2Y \leq k(k-1)^2 + 2\varepsilon + \max(2 - 2k, 4\varepsilon^2 - 2\varepsilon(k+1)).
\end{equation}
Also note that $2\varepsilon + 2 - 2k \leq - 2\varepsilon$ and that $2\varepsilon +4\varepsilon^2 - 2\varepsilon(k+1) = \underbrace{2\varepsilon(2\varepsilon - (k-1))}_{\leq0} - 2\varepsilon \leq -2\varepsilon$. Therefore, we have $2Y<k(k-1)^2-2\varepsilon$ for $\varepsilon < \frac{k-1}{2}$, since in that case the previous inequalities are strict.

Let us consider the case of $\varepsilon = \frac{k-1}{2}$ separately. In this case, if $2<|V_A| < 2\varepsilon$, then $2Y<k(k-1)^2-2\varepsilon$. As such, we have to consider only two values of $|V_A|$, namely $|V_A| = 2$ or $|V_A| = 2\varepsilon = k-1$.

Consider the $|V_A|=2\varepsilon = k-1$ case. By way of contradiction, assume that $2Y = k(k-1)^2-2\varepsilon$, instead of being strictly less than $k(k-1)^2-2\varepsilon$, which would imply that inequalities \eqref{eq:V_C_induced},~\eqref{eq:V_B_induced},~\eqref{eq:V_Ag},~\eqref{eq:2epsV_A} and \eqref{eq:E_AB} must all in fact  be equalities. We shall only focus on \eqref{eq:V_C_induced}.

From $|V_A| = 2\varepsilon$, we see that $|N(v)\cap N_2(u)| = k-1$, meaning there is only one vertex, say $x$, in $N(v)\setminus N_2(u)$. The contradiction will follow by noting the next three equalities. 

The first one is 
\begin{equation}\label{eq:case:k-1:V_C}
    V_C = N(x)\setminus \{v\} \Rightarrow |V_C| = k-1.
\end{equation}
First of all, we shall verify that $N(x)\setminus \{v\} \subseteq V_C$, noting that if there is $w \in N(x)\setminus (\{v\}\cup V_C)$, it means that $w \in N(x)\cap N_2(u)$, thus there is $i\in \{1, \ldots, k\}$ such that $w \in V_2(i)\cap N(v)$ making $d_i\geq 1$, which contradicts \eqref{eq:V_Ag}. And finally let us show that $N(x)\setminus \{v\}$ is exhaustive, meaning there is no $w \in N(v)\setminus\{x\}$ such that $N(w)\setminus (N_2(u)\cup N(u)\cup\{v\}) \neq \emptyset$. Indeed, if there were such a vertex $w \in N(v)\setminus\{x\}$, we would have an extra edge not adjacent to $v$ from $N_2(u)$, contradicting \eqref{eq:outerEdges}, since $|V_A| = 2\varepsilon = |E(N_2(u), \{v\})|$.

The second equality is 
\begin{equation}\label{eq:case:k-1:E_V_CV_C}
    E(V_C, V_C) = \emptyset,
\end{equation}
which holds since any vertex within $V_C$ would induce a $3$-cycle, given that $V_C = N(x)\setminus\{v\}$.

And the third equality is 
\begin{equation}\label{eq:case:k-1:E_V_BV_C}
    E(V_B, V_C) = \emptyset,
\end{equation}
which holds since otherwise we would have an extra edge not adjacent to $v$ from $N_2(u)$, contradicting \eqref{eq:outerEdges}, similarly to arguing the exhaustiveness of $N(v)\setminus \{x\}$.

Now, combining \eqref{eq:case:k-1:V_C},~\eqref{eq:case:k-1:E_V_CV_C} and~\eqref{eq:case:k-1:E_V_BV_C}, we see that 
\begin{equation}
    (k-1)\underbrace{|V_C|}_{k-1} > 2 \underbrace{|E(V_C, V_C)|}_{0} + \underbrace{|E(V_B, V_C)|}_{0}, 
\end{equation}
which clearly shows that inequality \eqref{eq:V_C_induced} can't be an equality, making the case $|V_A| = 2\varepsilon$ impossible.

Consider the case $|V_A| = 2$. Without loss of generality, we may assume $V_A = \{u_1, u_2\}$. Yet again, by way of contradiction, assume that $2Y = k(k-1)^2-2\varepsilon$, instead of being strictly less than $k(k-1)^2-2\varepsilon$, which would similarly imply that inequalities \eqref{eq:V_C_induced},~\eqref{eq:V_B_induced},~\eqref{eq:V_Ag},~\eqref{eq:2epsV_A} and \eqref{eq:E_AB} must, in fact, be equalities. We shall focus on \eqref{eq:2epsV_A} and~\eqref{eq:E_AB}, thus having
\begin{equation} \label{eq:d1_d2}
     d_1+d_2 = 2\varepsilon - 2. 
\end{equation}

Consider $v'' \in N(v)\cap V_2(2)$ and let us consider two separate cases depending on whether there is a vertex in $V_2(1)$ adjacent to~$v''$. It is worth noting that there must exist $i\in \{1, 3, 4,\ldots, k\}$ such that $v''$ is not adjacent to any of the vertices from $V_2(i)$.

Suppose $v''$ is not adjacent to any vertex in $V_2(i)$ for some $i>2$. Then there must be some vertex in $V_2(i)$ that is not adjacent to $V_2(2)$, forcing it to connect instead to $V(G)\setminus(N_2(u)\cup N(u))$. Similarly to the derivation of \eqref{eq:V_Ag}, this implies $d_1+d_2+1+2\leq 2\varepsilon$, which leads to a contradiction with \eqref{eq:d1_d2}.

Suppose $v''$ is not adjacent to any vertex in $V_2(1)$. Following the derivation of  \eqref{eq:d1_d2}, we can see that showing $|\{w \in V_B \cap V_2(1) : N(w) \cap N(v) \subset N_2(u)\}| = 0$ would suffice. By way of contradiction, assume that there is $w \in V_B\cap V_2(1)$ such that $N(w) \cap N(v) \subset N_2(u)$. This leads to a contradiction: we can clearly see that $N(w)\cap N(v)\cap N_2(u) = \emptyset$ due to the main premise of this case ($|V_A|=2$ and $v''$ is not adjacent to any vertex in $V_2(1)$) and the absence of $3$-cycles and $N(w)\cap N(v) \neq \emptyset$ due to $w$ being from $V_B$.

Therefore, we have established \eqref{eq:main_property}.

Similarly to before, consider an arbitrary vertex $u \in V(G)$ and $v \in N_3(u)$ with the goal of bounding from above the number of $5$-cycles containing the vertex $v$, which is denoted by $Y$. Let us also denote the only vertex in $N(v)\cap N_2(u)$ by $v'$ and the only vertex in $N(v')\cap N(u)$ by $u_1$ and say that $N(v)\setminus \{v'\} = \{v_1, \ldots, v_{k-1}\}$.
Then, let us define several sets  
\begin{align} \label{eq:definitions}
     V_A &= N_2(v)\cap N(u) = \{u_1\} \text{ (by \eqref{eq:main_property})},\\
     V_B &= N_2(v)\cap N_2(u), \;
     V_{B'} = V_B\cap N(v'),  \;
     V_{B''} = V_B\setminus N(v'),  \\
     V_C &= N_2(v)\setminus(N(u) \cup N_2(u)), \;
     V_{C'} = V_C\cap N(v'),     \;  V_{C''} = V_C\setminus N(v'),  \\
     L_i &= N(v_i)\cap{V_{C''}}, \; i=1, \ldots, k-1.
\end{align}

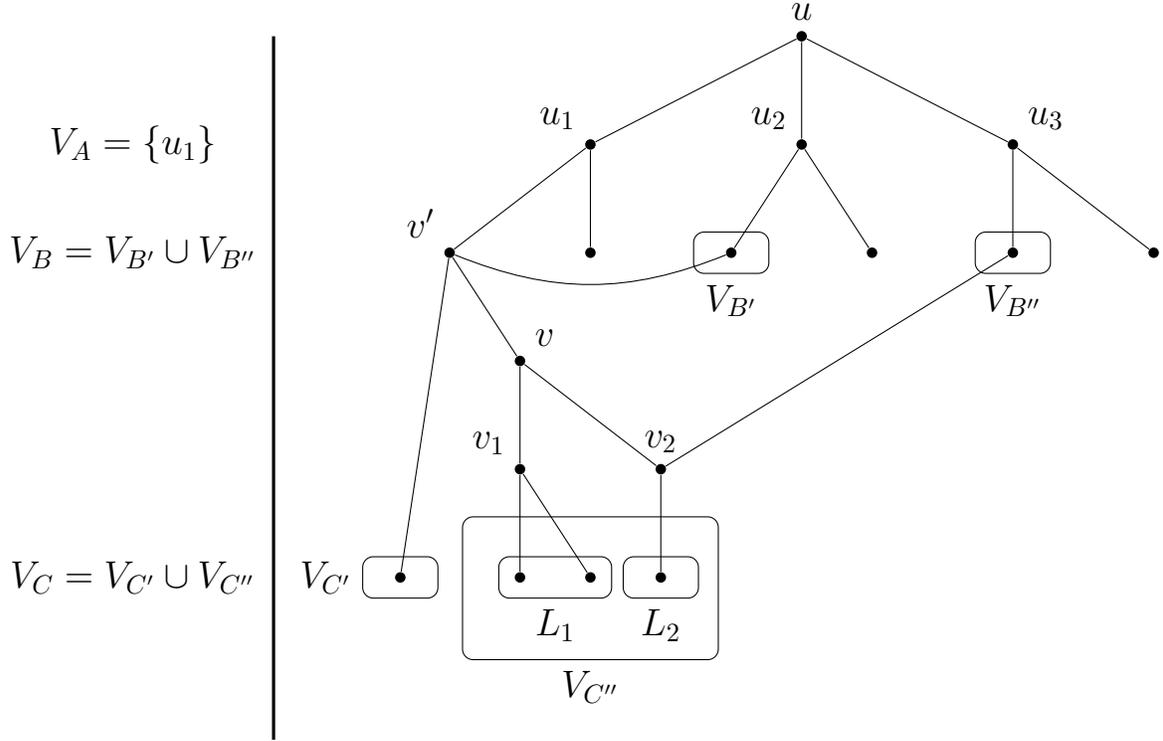
\begin{figure}[ht]
\centering
\begin{tikzpicture}[
    x=1.3cm,y=1cm,
    block/.style={draw,rounded corners,minimum width=1.0cm,minimum height=.55cm},
    dot/.style={circle,fill,inner sep=1.4pt},
    >=Latex,scale=0.72
]
    \node[dot,label=above:{\large $u$}] (u) at (0,0) {};

    \node[dot,label=above left:{\large $u_1$}] (uPrime) at (-3,-2) {};
    \node[dot,label=above left:{\large $u_2$}] (u1) at (0,-2) {};
    \node[dot,label=above right:{\large $u_3$}] (u2) at (3,-2) {};

    \node[dot,label=above left:{\large $v'$}] (vPrime) at (-5,-4) {};
    \node[dot] (v12) at (-3,-4) {};

    \node[dot] (v21) at (-1,-4) {};
    \node[dot] (v22) at (1,-4) {};

    \node[dot] (v31) at (3,-4) {};
    \node[dot] (v32) at (5,-4) {};
    
    \draw (u) -- (uPrime);
    \draw (u) -- (u1);
    \draw (u) -- (u2);

    \draw (uPrime) -- (vPrime);
    \draw (uPrime) -- (v12);

    \draw (u1) -- (v21);
    \draw (u1) -- (v22);

    \draw (u2) -- (v31);
    \draw (u2) -- (v32);

    \node[dot,label=above right:{\large $v$}] (v) at (-4,-6) {};
    \draw (v) -- (vPrime);

    \node[dot,label=above left:{\large $v_1$}] (v1) at (-4,-8) {};
    \node[dot,label=above:{\large $v_2$}] (v2) at (-2,-8) {};
    \draw (v) -- (v1);
    \draw (v) -- (v2);

    \node[dot] (v1Son1) at (-4,-10) {};
    \node[dot] (v1Son2) at (-3,-10) {};

    \draw (v1) -- (v1Son1);
    \draw (v1) -- (v1Son2);
    
    \node[dot] (v2Son1) at (-2,-10) {};
    \draw (v2) -- (v2Son1);

    \node[dot] (vPrimeSon) at (-5.7,-10) {};
    \draw (vPrime) -- (vPrimeSon);

    \draw (v2) -- (v31);
    \draw (vPrime) to[bend right=22] (v21);

    \node[block,label=below:{\large $V_{B''}$}] (D)  at (v31) {};
    \node[block,label=below:{\large $V_{B'}$}] (D)  at (v21) {};
    \node[block,label=left:{\large $V_{C'}$}] (D)  at (vPrimeSon) {};

    \node[block, minimum width=1.5cm, label=below:{\large  $L_1$}] at (-3.5,-10) {};
    \node[block, minimum width=1.0cm, label=below:{\large  $L_2$}] at (v2Son1) {};

    \node[block, minimum width=3.4cm, minimum height=1.9cm, label=below:{\large $V_{C''}$}] at (-3,-10.2) {};

    \draw[very thick] (-7.5,0) -- (-7.5,-13);

    \node at (-9.5,-2) {\large $V_A=\{u_1\}$};
    \node at (-9.5,-4) {\large $V_B=V_{B'} \cup V_{B''} $};
    \node at (-9.5,-10) {\large $V_C=V_{C'} \cup V_{C''} $};
\end{tikzpicture}
\caption{An illustration of the notation used for the case where exactly one vertex of $N_2(u)$ is adjacent to $v$. The displayed graph is a subgraph of $G$, so not all vertices and edges are visualized.}
\label{fig:illustrateNotation}
\end{figure}

We illustrate the notation that we use in Fig.~\ref{fig:illustrateNotation}. The main point of introducing all these sets in \eqref{eq:definitions} is to improve inequalities \eqref{eq:V_C_induced} and \eqref{eq:V_Ag}. We shall do this by careful consideration of the number of edges between these sets, which naturally produce many secondary inequalities. As such, we shall list the important inequalities, show that if they hold true, then $2Y < k(k-1)^2 - 2\varepsilon$, and after having done it, justify the inequalities.

In the same vein, we see that 
\begin{equation}    
      Y =  |E(V_A, V_B)|  + |E(V_B, V_B)| + |E(V_B, V_C)| + |E(V_C, V_C)|   
\end{equation}
and that \eqref{eq:V_B_induced} holds.
 
We now state the inequalities \eqref{eq:V_C_induced_new},  \eqref{eq:new_outer_bound1} and \eqref{eq:new_outer_bound2}, which we will later prove. The improved version of \eqref{eq:V_C_induced} is the following inequality
\begin{multline} \label{eq:V_C_induced_new}    
    (k-1)|V_{C}| \geq 2|E(V_C, V_C)| + |E(V_B, V_C)| \\ + (k-1)(k-1-|V_{C'}|) - |V_{B''}| - |E(V_{C''}, V_B)|.
\end{multline} Moreover, \eqref{eq:V_Ag} splits into 
\begin{equation} \label{eq:new_outer_bound1}
    |V_{B''}| + |V_{C'}| + 1  + |E(V_B, V_{C''})|\leq 2\varepsilon 
\end{equation}
and 
\begin{equation}\label{eq:new_outer_bound2}
    |V_{B''} \cap N(u_1)| + |V_{C'}| + 1 \leq \varepsilon. 
\end{equation}
We can readily see that \eqref{eq:E_AB} translates into
\begin{equation}\label{eq:eq:E_AB_new}
    |E(V_A, V_B)| = |V_{B''}\cap N(u_1)| 
\end{equation}
since $V_A=\{u_1\}$ and $u_1$ cannot have a neighbor in $V_{B'}$ because $G$ has girth $5$.

Thus, combining \eqref{eq:V_B_induced},~\eqref{eq:V_C_induced_new},~\eqref{eq:new_outer_bound1},~\eqref{eq:new_outer_bound2} and \eqref{eq:eq:E_AB_new} we see that 
\begin{multline}\label{eq:finshing_Y}    
    2Y \leq (k-1)(\underbrace{|V_B|+|V_C|}_{k(k-1)-1}) - (k-1)(k-1 - |V_{C'}|) \\ + \underbrace{|V_{B''}| + |E(V_{C''}, V_B)|}_{\leq 2\varepsilon - 1 - |V_{C'}|}+ \underbrace{|V_{B''}\cap N(u_1)|}_{\leq \varepsilon - |V_{C'}| - 1} \\ \leq k(k-1)^2 - (k-1)(k - |V_{C'}|) + 3\varepsilon - 2 |V_{C'}| - 2. 
\end{multline}

We should note that from \eqref{eq:new_outer_bound2} we see that $|V_{C'}|\leq \varepsilon - 1$, thus allowing us to maximize the right hand side of \eqref{eq:finshing_Y} to see that
\begin{equation}
    2Y\leq k(k-1)^2 - (k-1)(k+1-\varepsilon) + \varepsilon < k(k-1)^2 - 2\varepsilon.
\end{equation}
Here, the last inequality holds since $\varepsilon \leq \frac{k-1}{2}$ implies that $\varepsilon < \frac{k^2-1}{k+2}$ for $k\geq 3$, since $\frac{2(k+1)}{k+2}>1$.

It therefore suffices to prove the inequalities \eqref{eq:V_C_induced_new},  \eqref{eq:new_outer_bound1} and \eqref{eq:new_outer_bound2} to show that the theorem holds.

Let us start by showing \eqref{eq:new_outer_bound1}. To that end, we shall establish the following set of identities
\begin{align}
    &|V_{C'}| = |E(V_{C'}, N_2(u))|, \label{eq:Cprime_to_edges}\\
    &|V_A| = 1 = |E(\{v\}, N_2(u))|, \label{eq:A_to_edges}\\
    &|E(V_B, V_{C''})| \leq |E(V_{C''}, N_2(u))|, \label{eq:general_edges} \\
    &|V_{B''}| = |E(N(v)\setminus\{v'\}, N_2(u))|, \label{eq:Vprimeprime_to_edges}
\end{align}
which combined with \eqref{eq:outerEdges} establishes \eqref{eq:new_outer_bound1}.

In order to establish \eqref{eq:Cprime_to_edges} and \eqref{eq:A_to_edges} we note that 
\begin{equation}\label{eq:important_absence}
 E(V_{C'}\cup\{v\}, N_2(u)\setminus\{v'\}) = \emptyset.   
\end{equation}
Indeed, assume, for the sake of contradiction, that there is an edge between $w \in N_2(u)\setminus\{v'\}$ and $w' \in V_{C'}\cup\{v\}$, making $w' \in N(v')$. Thus, we can see that $w, v' \in N_2(u)$ and $w' \in N(w)\cap N(v')\neq\emptyset$. Thus, by virtue of \eqref{eq:main_property} we see that $w$ must equal $v'$, but $w\neq v'$ by definition. A simple corollary of \eqref{eq:important_absence} is that $|E(V_{C'}, N_2(u))| =|V_{C'}|$ and $|E(\{v\}, N_2(u))|=1$.

It is obvious that inequality \eqref{eq:general_edges} follows from monotonicity, using $V_B \subset N_2(u)$. 

Equation \eqref{eq:Vprimeprime_to_edges} follows from the fact that $\forall w\in V_{B''}\exists!v_s\in N(v)\setminus\{v'\}:v_s w\in E$ inducing an injective function $$f_{B''} : V_{B''} \rightarrow E(N(v)\setminus\{v'\}, N_2(u))$$ and the fact that $f_{B''}$ is surjective. The existence follows from the fact that $w\in V_{B''}$ implies $w \in N_2(v)\setminus N(v')$. For the sake of contradiction, assume that for some $w\in V_{B''}$ there are two distinct vertices $v_{s_1}$ and $v_{s_2}$ in $N(v)\setminus\{v'\}$ such that $\{v_{s_1}w, v_{s_2}w\} \subset E$, but then $v, v_{s_1}, w, v_{s_2}$ makes a $4$-cycle, which contradicts the fact that $G$ has girth $5$. Finally, assume that $f_{B''}$ is not surjective, meaning there must be $v_{s} \in N(v)\setminus \{v'\}$ and $w \in N_2(u)$ such that $v_{s}w \in E\setminus f_{B''}(V_{B''})$. We can note that $w$ must be in $V_{B''}$, since $w \in N_2(u)$ and $w\in N_2(v)$ given $v_{s} \in N(v)$. Assume, for the sake of contradiction, that $w \in N(v')$. This assumption introduces a $4$-cycle $v, v', w, v_s$. Thus $w \in V_{B''}$, which shows that $f_{B''}$ is bijective.

Therefore, as stated before, we establish \eqref{eq:new_outer_bound1}
\begin{multline}
    |V_{B''}| + |V_{C'}| + 1  + |E(V_B, V_{C''})| \leq |V_{B''}| + |V_{C'}| + 1  + |E(V_{C''}, N_2(u))| \\ =  |E(N(v)\setminus\{v'\}, N_2(u))|  + |E(V_{C'}, N_2(u))| +  |E(\{v\}, N_2(u))| \\ +  |E(V_{C''}, N_2(u))| \leq |E(V\setminus(N(u)\cup N_2(u)), N_2(u))| = 2\varepsilon.
\end{multline}

Let us shift our attention towards proving \eqref{eq:new_outer_bound2}, to that end, we shall note that 
\begin{align}
    &|V_{C'}| = |E(V_{C'}, N_2(u)\cap N(u_1))|, \label{eq:Cprime_to_edges_more}\\
    &|V_A| = 1 = |E(\{v\}, N_2(u)\cap N(u_1))|, \label{eq:A_to_edges_more}\\
    &|V_{B''} \cap N(u_1)| = |E(N(v)\setminus\{v'\}, N_2(u)\cap N(u_1))|, \label{eq:Vprimeprime_to_edges_more}
\end{align}
which are just refined versions of \eqref{eq:Cprime_to_edges},~\eqref{eq:A_to_edges} and~\eqref{eq:Vprimeprime_to_edges}. Indeed, \eqref{eq:Cprime_to_edges_more} and~\eqref{eq:A_to_edges_more} hold due to \eqref{eq:important_absence}. And equation \eqref{eq:Vprimeprime_to_edges_more} holds because we can see, using a similar argument as before, that $$f_{B''}|_{V_{B''}\cap N(u_1)}: V_{B''}\cap N(u_1) \rightarrow E(N(v)\setminus\{v'\}, N_2(u)\cap N(u_1))$$ is a bijection.

The next step is to establish that 
 \begin{equation} \label{eq:another_edge_count}
  |E(N_2(u)\cap N(u_1), V\setminus(N(u)\cup N_2(u)))| \leq \varepsilon.   
 \end{equation}
To that end, we shall name vertices $N(u)\setminus\{u_1\}$ as $\{u_2, \ldots, u_{k}\}$ and construct a new multi-graph $G'$ with vertices $u_1, u_2, u_3, \ldots, u_{k}$ such that any two vertices $u_i, u_j \in V(G')$ are connected with edge multiplicity $|E(N_2(u)\cap N(u_i), N_2(u)\cap N(u_j))|\leq k-1$. Then we can see that $|E(G')| = \frac{k(k-1)^2}{2} - \varepsilon$, and also we observe that \eqref{eq:another_edge_count} is equivalent to the claim that $\deg_{G'}(u_1) \geq (k-1)^2-\varepsilon$, which holds since otherwise it would mean that $u_1$ misses at least $\varepsilon + 1$ edges, but there are only $\varepsilon$ of those in $G'$. 

Therefore, we establish \eqref{eq:new_outer_bound2} by combining \eqref{eq:Cprime_to_edges_more},~\eqref{eq:A_to_edges_more},~\eqref{eq:Vprimeprime_to_edges_more} and~\eqref{eq:another_edge_count}
\begin{multline}
    |V_{C'}| + 1 + |V_{B''}\cap N(u_1)| = |E(V_{C'}, N_2(u)\cap N(u_1))| + |E(\{v\}, N_2(u)\cap N(u_1))| \\ + |E(N(v)\setminus\{v'\}, N_2(u)\cap N(u_1))| \leq  |E(N_2(u)\cap N(u_1), V\setminus(N(u)\cup N_2(u)))| \leq \varepsilon.
\end{multline}

Let us focus on \eqref{eq:V_C_induced_new}. To that end, mirroring the derivation of \eqref{eq:V_C_induced}, having denoted $V(G)\setminus(V_C \cup V_B \cup V_A)$ by $V_{\text{Out}}$, let us note that
\begin{multline}\label{eq:V_C_prime_induced}
     (k-1)|V_{C'}| = \sum_{w \in V_{C'}} |E(\{w\}, V(G)\setminus \{v'\})| \geq \sum_{w \in V_{C'}} |E(\{w\}, V_{C'} \cup V_{C''} \cup V_B)| \\ = 2 |E(V_{C'}, V_{C'})| + |E(V_{C'}, V_{C''})| + |E(V_B, V_{C'})|   
\end{multline}
and that
\begin{multline}\label{eq:V_C_double_prime_induced}
     (k-1)|V_{C''}| = \sum_{w \in V_{C''}} |E(\{w\}, V(G)\setminus N(v))| \geq \sum_{w \in V_{C''}} |E(\{w\}, V_{C'} \cup V_{C''} \cup V_B \cup V_{\text{Out}})| \\ = 2 |E(V_{C''}, V_{C''})| + |E(V_{C'}, V_{C''})| + |E(V_B, V_{C''})| + |E(V_{C''}, V_{\text{Out}})|,  
\end{multline}
transforming \eqref{eq:V_C_induced} into 
\begin{equation}
    (k-1)|V_C| \geq 2 |E(V_C, V_C)| + |E(V_B, V_C)|  + |E(V_{C''}, V_{\text{Out}})|.
\end{equation}

Moving forward, we shall rely on the following inequalities
\begin{align}
&|E(L_i, V_{C'})| \leq |V_{C'}|,\label{eq:LCprime}\\
&|E(L_i, V_{C''})| \leq |L_i|(k-2).\label{eq:LCdoubleprime}\\  
\end{align}

For the sake of contradiction, assume that \eqref{eq:LCprime} fails to hold, meaning $|E(L_i, V_{C'})| > |V_{C'}|$, which implies the existence of $w \in V_{C'}$ such that there are two distinct vertices $w'_1, w'_2 \in L_i$ adjacent to $w$. This yields the following $4$-cycle $w, w'_1, v, w'_2$.

To establish \eqref{eq:LCdoubleprime}, we shall note that $E(L_i, L_i) = \emptyset$, otherwise it would create $3$-cycles and that $\forall w\in L_i:|E(\{w\}, L_j)| \leq 1, i\neq j$, or else there would be $w'_1, w'_2 \in L_j$ adjacent to $w$ producing a $4$-cycle $w'_1, v, w'_2, w$. Thus making
\begin{equation}
    |E(L_i, V_{C''})| = \sum_{j\neq i} |E(L_i, L_j)| = \sum_{w \in L_i} \sum_{j\neq i} \underbrace{|E(\{w\}, L_j)|}_{\leq 1} \leq \sum_{w \in L_i} \sum_{j\neq i} 1=  |L_i|(k-2).
\end{equation}

We also establish the following identity, involving $V_{\text{Out}}$
\begin{multline} \label{eq:expansion_of_Li}   
    (k-1)|L_i| = \sum_{w \in L_i} |E(\{w\}, V(G)\setminus\{v_i\})|= 2\underbrace{|E(L_i, L_i)|}_{0} \\ +  \underbrace{|E(L_i, V_{C''})|}_{|E(L_i, V_{C''}\setminus L_i)|} + |E(L_i, V_{C'})| + |E(L_i, V_B)| + |E(L_i, V_{\text{Out}})|.
\end{multline}

Let us now take a close look at $|E(V_{C''}, V_{\text{Out}})|$, making use of  \eqref{eq:LCprime},~\eqref{eq:LCdoubleprime} and~\eqref{eq:expansion_of_Li}.
\begin{multline}\label{eq:Vout_bound}
    |E(V_{C''}, V_{\text{Out}})| +  |E(V_{C''}, V_{B})| = |E(V_{C''}, V_B \cup V_{\text{Out}})|  \\ = \sum_{i} |E(L_i, V_B \cup V_{\text{Out}})|  = \sum_{i} \left((k-1)|L_i| - |E(L_i, V_{C''})| - |E(L_i, V_{C'})|\right) \\ \geq \sum_{i} \left((k-1)|L_i| - |L_i|(k-2) - |V_{C'}|\right) = \sum_{i} \left(|L_i| - |V_{C'}|\right).
\end{multline}

Let us investigate the inner structure of $L_i$. We claim that 
\begin{equation}\label{eq:Liproperty1}
\forall i \in \{1, \ldots, k-1\}: |L_i| \geq k-2.
\end{equation}
For the sake of contradiction, assume that $|L_i|\leq k-3$. Then, there are at least two vertices $w_1, w_2 \in N(v_i)\setminus (V_{C''}\cup \{v\})$. We can see that neither of these vertices is in $V_{C'}$, since otherwise it would induce a $4$-cycle $w_j, v', v, v_i, j\in \{1, 2\}$, and neither can be in $V_A$. Thus, both $w_j$ have to be in $V_B \subset N_2(u)$, but then $v_i \in N(w_1)\cap N(w_2)$ and $v_i \notin N(u)\cup N_2(u)$, which is a contradiction to \eqref{eq:main_property}.

Moreover, we can prove that the number of $L_i$'s such that their cardinality is $k-2$ equals exactly $|V_{B''}|$. Indeed, as we established in the previous paragraph, if $|L_i| = k-2$, then there is exactly one vertex $w \in N(v_i)\cap V_B$. We can quickly note that $w \notin V_{B'}$, since otherwise there would be a $4$-cycle $w, v', v, v_i$. And also, by definition of $V_{B''}$, for every vertex $w \in V_{B''}$ it is true that there exists $i \in \{1, \ldots, k-1\}$ such that $w \in N(v_i)$. Moreover, there is only one such $i$, because otherwise if additionally $w \in N(v_j)$, then we would have a $4$-cycle $w, v_i, v, v_j$, proving the claim.

Now, we can refine the upper bound in \eqref{eq:Vout_bound} to the following
\begin{multline}    
    |E(V_{C''}, V_{\text{Out}})| +  |E(V_{C''}, V_{B})| \geq \sum_{i} \left(|L_i| - |V_{C'}|\right) \\ = |V_{B''}|(k-2 - |V_{C'}|) + (k-1 - |V_{B''}|)(k-1 - |V_{C'}|) \\ =  (k-1 )(k-1 - |V_{C'}|) - |V_{B''}| .
\end{multline}

This concludes the proof.

\end{proof}

\section*{Acknowledgements}
\noindent We are grateful to Yurii Yarosh for useful discussions. Jorik Jooken is supported by a Postdoctoral Fellowship of the Research Foundation Flanders (FWO) with grant number 1222524N. 

\bibliographystyle{abbrv}
\typeout{}
\bibliography{references}

\end{document}